\documentclass[12pt,reqno]{amsart}
\usepackage{graphicx}
\usepackage{amssymb}

\usepackage[bookmarks,
colorlinks,
]{hyperref}

\numberwithin{equation}{section}
\newtheorem{thm}{Theorem}[section]
\newtheorem{cor}[thm]{Corollary}
\newtheorem{lem}[thm]{Lemma}
\newtheorem{prop}[thm]{Proposition}

\theoremstyle{definition}
\newtheorem{defn}[thm]{Definition}
\newtheorem{exmp}[thm]{Example}
\newtheorem{rem}[thm]{Remark}

\numberwithin{equation}{section}

\newcommand\Hom{\operatorname{Hom}}

\newcommand\Ext{\operatorname{Ext}}

\newcommand{\Cech}{{{\Check{C}}}_{\xx}}
\newcommand\cd{\operatorname{cd}}
\newcommand\Rad{\operatorname{Rad}}
\newcommand\Ann{\operatorname{Ann}}

\newcommand\Ass{\operatorname{Ass}}
\newcommand\Supp{\operatorname{Supp}}

\newcommand\Tor{\operatorname{Tor}}
\newcommand{\xx}{\underline x}

\newcommand{\qism}{\stackrel{\sim}{\longrightarrow}}
\newcommand\grade{\operatorname{grade}}
\newcommand\depth{\operatorname{depth}}

\newcommand\im{\operatorname{im}}

\begin{document}

\title[Endomorphisms of local cohomology]{A note on endomorphisms of local cohomology modules}

\begin{abstract} Let $I$ denote an ideal of a local ring $(R,\mathfrak{m})$ of dimension $n$. Let
$M$ denote a finitely generated $R$-module. We study the endomorphism ring of the
local cohomology module $H^c_I(M), c = \grade (I,M)$. In particular there is a
natural homomorphism $\Hom_{\hat{R}^I}(\hat{M}^I, \hat{M}^I)\to \Hom_{R}(H^c_{I}(M),H^c_{I}(M))$,
where $\hat{\cdot}^I$ denotes the $I$-adic completion functor. We prove sufficient
conditions such that it becomes an isomorphism. Moreover, we study a homomorphism
of two such endomorphism rings of local cohomology modules for two ideals $J \subset I$
with the property $\grade(I,M) = \grade(J,M)$. Our results extends
constructions known in the case of $M = R$ (see e.g.  \cite{h1}, \cite{p7}, \cite{p1}).
\end{abstract}

\author[W. Mahmood]{Waqas Mahmood}
\address{Quaid-I-Azam University, Islamabad Pakistan}%
\email{waqassms$@$gmail.com}

\author[Z. Zahid]{Zohaib Zahid}
\address{Abdus Salam School of Mathematical Sciences, GCU, Lahore Pakistan}
\email{Zohaib\_zahid@hotmail.com}

\thanks{This research was partially supported by Higher Education Commission, Pakistan}

\maketitle

\section{Introduction}
Let $(R,\mathfrak m,k)$ be a local Noetherian ring. For an ideal $I\subset R$ and an $R$-module $M$ we
denote by $H^i_I(M)$, $i\in \mathbb{Z}$, the local cohomology modules of $M$ with
respect to $I$ (see \cite{b} and \cite{goth} for the definition). In recent papers
(see e.g. \cite{e}, \cite{h1}, \cite{pet1}, \cite{p7}, \cite{p1}, \cite{p6}) there is some interest
in the study of the endomorphism ring $\Hom_R(H^i_I(R),H^i_I(R))$ for certain ideals
$I$ and several $i \in \mathbb{N}$. Note that the first results in this direction
were obtained by M. Hochster and C. Huneke (see \cite{HH}) for the case of $I =
\mathfrak{m}$ and $i = \dim R$.

In the case of $(R, \mathfrak{m})$ an $n$-dimensional Gorenstein ring and $I \subset R$ an ideal with
$c = \grade I$ it was shown (see \cite[Theorem 1.1]{p6}) that there is a natural homomorphism
\[
\hat{R}^I \to \Hom_R(H^c_I(R),H^c_I(R))
\]
where $\hat{R}^I $ denotes the $I$-adic completion of $R$. Moreover there are results when this homomorphism is in fact an isomorphism (see
also \cite{h1}, \cite{p7}). The main subject of the present paper
is the extension of some of these results to the case of a finitely generated $R$-module $M$.
More precisely we shall prove the following result:

\begin{thm} \label{1.1} Let $I$ denote an ideal of $R$ with $\dim(R)=n$. Let $M$ denote a finitely generated
$R$-module and $c = \grade (I,M)$. Then there is a natural homomorphism
\[
\Hom_{\hat{R}^I}(\hat{M}^I, \hat{M}^I)\to \Hom_{R}(H^c_{I}(M),H^c_{I}(M)).
\]
Suppose that $H^i_I(M) = 0$ for all $i \not=c$ then it is an isomorphism. Moreover
\[
\Ext^i_{\hat{R}^I}(\hat{M}^I,\hat{M}^I) \cong \Ext^{i+c}_R(H^c_I(M), M)\cong \Ext^{i}_R(H^c_I(M), H^c_I(M))
\]
for all $i \in \mathbb{Z}$.
\end{thm}

That is, the endomorphism rings of $M$ and those of $H^c_I(M)$ are closely
related by a natural map (see Theorem \ref{2} and Corollary \ref{3a} for the proof).
These investigations are related to the endomorphism ring
of $D(H^c_I(M)):= \Hom_R(H^c_I(M), E), E$ is the injective hull of the residue field $k = R/\mathfrak{m}$.
Let $(R,\mathfrak{m})$ denote a complete local ring of dimension $n$. Then there is a natural homomorphism
\[
\Hom_{R}(H^c_{I}(M),H^c_{I}(M)) \to \Hom_{R}(D(H^c_{I}(M)),D(H^c_{I}(M)))
\]
that is an isomorphism (see Lemma \ref{3}). Moreover there are relative versions
of the above homomorphisms relating two ideals $J \subset I$ of the same grade(see Theorem \ref{21.3}). Also there are some necessary conditions in order to prove the natural homomorphism $\Hom_{R}(M,M)\to \Hom_{R}(H^c_{I}(M),H^c_{I}(M))$ is an isomorphism. That is

\begin{thm} \label{1.2} Let $J \subset I$ denote two ideals of a complete local ring $(R, \mathfrak{m})$ of $\dim(R)=n$.
Let $M$ be a finitely generated $R$-module with $H^i_J(M)= 0$ for all $\neq c = \grade (I,M) = \grade(J,M)$. Suppose that $\Rad IR_{\mathfrak{p}}= \Rad JR_{\mathfrak{p}}$ for all ${\mathfrak{p}}\in V(J)\cap \Supp_R(M)$ such that $\depth_{R_\mathfrak{p}} (M_{\mathfrak{p}})\leq c$.
Then the natural homomorphism
\[
\Hom_{R}(M,M)\to \Hom_{R}(H^c_{I}(M),H^c_{I}(M))
\]
is an isomorphism.
\end{thm}
This is a generalization of \cite[Theorem 1.2]{p7}) to the case of a finitely generated $R$-module $M$.
In the case of an ideal $J$ generated by an $M$-regular sequence of length $c$ in a complete
local ring $R$ it follows that $\Hom_R(M,M) \cong \Hom_{R}(H^c_{J}(M),H^c_{J}(M))$. So the above result
gives some necessary conditions for the endomorphism ring of $H^c_I(M)$ to be isomorphic
to that of $M$. The proof of Theorem \ref{1.2} is shown in \ref{0}.

\section{Preliminaries and auxiliary results}
In this section we will fix some notation  and summarize a few preliminaries and auxiliary results.
For unexplained terminologies we refer to the textbooks \cite{Ma} and \cite{w}.
Let $(R,\mathfrak m)$ denote a commutative local Noetherian ring with $\mathfrak m$ its maximal ideal
and $k= R/{\mathfrak m}$ its residue field. Moreover we will denote the Matils dual functor by
$D(\cdot):= \Hom_R(\cdot, E)$ where $E= E_R(k)$ denotes the injective hull of $k$.

\begin{lem}\label{4.5}
Let $M,N$ be two arbitrary $R$-modules. Then  for all $i\in \mathbb{Z}$:
\begin{itemize}
\item[(1)] $\Ext^{i}_R(N,D(M))\cong D(\Tor_{i}^R(N, M))$.

\item[(2)] If $N$ is in addition finitely generated then
\[
D(\Ext^{i}_R(N,M))\cong \Tor_{i}^R(N, D(M)).
\]
\end{itemize}
\end{lem}

\begin{proof} The proof is well known for details see e.g. \cite[Example 3.6]{h}.
\end{proof}

Let $I$ be an ideal of $R$ and let $M$ denote an $R$-module. For the basics on local
cohomology modules $H^i_I(M)$ we refer to the textbook \cite{b}. Note that
\[
\grade(I,M) = \inf\{ i \in \mathbb{Z} : H^i_I(M) \not= 0\}
\]
for a finitely generated $R$-module $M$.
As a first result we will give a criterion of calculating the $\grade(I,M)$  in terms of
vanishing of the $\Tor$-modules.

\begin{prop}\label{31}
Let $I \subset R$ be an ideal. Let $M$ denote a finitely generated $R$-module with $c= \grade(I, M)$.
Let  $N$ be an  $R$-module such that $\Supp_R(N)\subseteq V(I)$. Then the following holds:
\begin{itemize}
  \item [(a)] $\Ext^{i}_R(N, M)=0$ for all $i < c$ and there is a natural isomorphism
\[
\Ext^{c}_R(N, M)\cong \Hom_R(N, H^c_I(M)).
\]
  \item [(b)] $\Tor_{i}^R(N, D(M))= 0 $ for all $i < c$ and there is a natural isomorphism
  \[
\Tor_{c}^R(N, D(M))\cong N\otimes_R D(H^c_I(M)).
\]
\end{itemize}
\end{prop}

\begin{proof} In the case of $M = R$ the result was shown by Schenzel (see \cite[Thoerem 2.3]{p1}).
For a finitely generated $R$-module $M$ the statement follows by the same arguments. So,
we skip the details.
\end{proof}

As an application of Proposition \ref{31} we will give a characterization of $\grade(I,M)$.
For $M = R$ this is shown in \cite[Corollary 2.4]{p1}.

\begin{cor}
With the notation of Proposition \ref{31} it follows
\[
\grade(I, M)= \inf\{i\in \mathbb{Z}: \Tor_{i}^R(R/I, D(M))\neq 0\}.
\]
\end{cor}

\begin{proof}
By Proposition \ref{31} it will be enough to show that $\Tor_{c}^R(R/I, D(M))\neq 0$.
Since $\grade(I,M)= c$  it follows  that (see Proposition \ref{31})
\[
\Tor_{c}^R(R/I, D(M))\cong R/I\otimes_R D(H^c_I(M)).
\]
This  module is isomorphic to $D(\Hom_R(R/I, H^c_I(M)))$ (see Lemma \ref{4.5}). Therefore
 we conclude that (see Proposition \ref{31})
\[
\Tor_{c}^R(R/I, D(M))\cong D(\Ext^c_R(R/I, M))\neq 0.
\]
Since $\Ext^c_R(R/I, M) \not= 0$ because of $ c = \grade(I,M)$ (see \cite[Remark 3.11]{h}) it completes the proof.
\end{proof}
As a final point of this section let us recall the definition of the truncation complex as it
was introduced in \cite[Definition 4.1]{p4}. Let $(R,\mathfrak m)$ be a local ring of dimension $n$ and let $M$
be a finitely generated $R$-module. Let $I\subset R$
be an ideal of $R$ with $\grade(I, M) = c$. Suppose that $E^{\cdot}_R(M)$ is a minimal injective
resolution of $M.$ Then it follows  (see Matlis' structure theory on injective modules \cite{m}) that
\[
E^{\cdot}_R(M)^i\cong\bigoplus\limits_{{\mathfrak p}\in \Supp_R(M)}E_R(R/{\mathfrak p})^{\mu_{i}({\mathfrak p}, M)}.
\]
where $\mu_{i}({\mathfrak p}, M)= \dim_k(\Ext_R^i(k,M))$ and $i\in \mathbb{Z}$. Let $\Gamma_I(\cdot)$ denote the section functor with support in $I$. Then $\Gamma_I(E^{\cdot}_R(M))^i= 0$ for all $i< c$ (see \cite[Definition 4.1]{p4} for more details). Whence there
is an exact sequence
\[
0 \to H^c_I(M) \to \Gamma_I(E^c_R(M)) \to \Gamma(E^{c+1}_R(M)).
\]
That is there is a natural morphism of complexes of $R$-modules $H^c_I(M)[-c] \to \Gamma(E^{\cdot}_R(M))$, where
$H^c_I(M)$ is considered as a complex sitting in homological degree 0.

\begin{defn} \label{2.2}
Define the complex $C^{\cdot}_M(I)$ as the cokernel of the above morphism of complexes.
It is called the truncation complex of $M$ with respect to the ideal $I$. There is a short exact sequence of complexes of $R$-modules
\[
0\rightarrow H^c_I(M)[-c]\rightarrow \Gamma_I(E^{\cdot}_R(M))\rightarrow C^{\cdot}_M(I)\rightarrow 0.
\]
\end{defn}

Note that $H^i(C^{\cdot}_M(I)) = 0$ for all $i\leq c$ or $i> n$ and $H^i(C^{\cdot}_M(I)) \cong H^i_I(M)$ for all $c< i\leq n$.
As a first application there is the following Proposition.

\begin{prop} \label{ps2} Let $I \subset R$ denote an ideal. Let $M$ denote a finitely
generated $R$-module. Suppose that $H^i_I(M) = 0$ for all $i \not= c = \grade(I,M)$. Then there
are natural isomorphisms $\Ext^i_R(N,M) \cong \Ext^{i-c}_R(N,H^c_I(M))$ for all $i \in \mathbb{Z}$
and any $R$-module $N$ with $\Supp_R N \subseteq V(I)$.
\end{prop}

\begin{proof} Let $E_R^{\cdot}(M)$ denote a minimal injective resolution of $M$. By the assumption
it follows that $\Gamma_I(E^{\cdot}_R(M))$ is a minimal injective resolution of $H^c_I(M)[-c]$.
Therefore there is an isomorphism
\[
\Ext^{i-c}_R(N,H^c_I(M)) \cong H^i(\Hom_R(N,\Gamma_I(E^{\cdot}_R(M))))
\]
for all $i \in \mathbb{Z}$. By an application of \cite[Lemma 2.2]{p1} it follows that
\[
\Hom_R(N,\Gamma_I(E^{\cdot}_R(M))) \cong \Hom_R(N,E^{\cdot}_R(M))
\]
for an $R$-module $N$
such that $\Supp_R N \subseteq V(I)$. Because of
\[
\Ext^i_R(N,M) = H^i(\Hom_R(N, E^{\cdot}_R(M)))
\]
the result follows.
\end{proof}

\section{A natural map of endomorphism rings}
In this section $I$ denotes an ideal of a local ring $(R,\mathfrak{m})$. Let $M$ be an
$R$-module. We denote by $\hat{M}^I = \Lambda^I(M) = \varprojlim M/I^{\alpha}M$ the
$I$-adic completion of $M$. In particular $\hat{R}^I$ denotes the $I$-adic completion
of $R$. The natural homomorphism $R \to \hat{R}^I$ makes $\hat{R}^I$ into a faithfully flat
$R$-module. We denote by $\Lambda^I(\cdot)$ the completion functor. Its left
derived functors are denoted by ${\rm L}_i\Lambda^I(\cdot), i \in \mathbb{Z}$. Note that
${\rm L}_i\Lambda^I(M) = 0$ for all $i > 0$ and a finitely generated $R$-module $M$ as
follows since $\Lambda^I(\cdot)$ is exact on the category of finitely generated $R$-modules.

Before we shall prove the main theorem we need another result on completion. To this end
let $\underline{x} = x_1,\ldots,x_r\in I$ denote a system of elements such that $\Rad (\underline{x})R
= \Rad I$. Then we consider the \v{C}ech complex $\Cech$
of the system $\xx$ (see e.g. \cite[Section 3]{pS2}).
Moreover, there is a bounded complex
$L_{\xx}$ of free $R$-modules (depending on $\xx = x_1,\ldots,x_r$) and a
natural homomorphism of complexes $L_{\xx} \to \Cech$  that induces an isomorphism
in cohomology (see \cite[Section 4]{pS2} for the details). That is, $L_{\xx}$ provides
a free resolution of $\Cech$, a complex of flat $R$-modules.
In our context the importance of the \v{C}ech complex
is the following result. It allows the expression of the left derived
functors ${\rm L}_i\Lambda^I(\cdot)$ in different terms.

\begin{thm} \label{2.3} Let $I$ denote an ideal of
ring $(R,\mathfrak{m})$. Let $X^{\cdot}$ denote a complex of flat $R$-modules. Then there are natural isomorphisms
\[
{\rm L}_i\Lambda^I(X^{\cdot}) \cong H_i(\Hom_R(L_{\xx},X^{\cdot})) \cong
H_i(\Hom_R(\Cech,E^{\cdot}) )
\]
for all $i \in \mathbb{Z}$, where $E^{\cdot}$ denotes an injective resolution of $X^{\cdot}$.
\end{thm}

\begin{proof} See \cite{ALL}, \cite{psy} or \cite[Theorem 1.1]{pS2}.
\end{proof}

The last Theorem \ref{2.3} is one of the main results of \cite[Section 4]{pS2} in the case of a bounded complex. The general version was proved by  M.~Porta, L.~Shaul, A.~Yekutieli (see \cite{psy}).
In our formulation here we do not use the technique of derived functors and derived categories.
(For a more advanced exposition based on derived categories and derived functors the
interested reader might also consult \cite{ALL} and \cite{psy}.)

\begin{thm}\label{2}
Let $(R,\mathfrak m)$ denote a local ring with $\dim(R)=n$. Let $M$ be a finitely generated $R$-module. Let $I\subset R$ be an ideal and $c= \grade (I,M)$.  Then there is a natural homomorphism
\[
\Hom_{\hat{R}^I}(\hat{M}^I, \hat{M}^I)\to \Hom_{R}(H^c_{I}(M),H^c_{I}(M)).
\]
\end{thm}

\begin{proof}
Let $E^{\cdot}_R(M)$ be a minimal injective resolution of $M$.  Then we use the truncation
complex as defined in \ref{2.2}. We apply the functor $\Hom_R(\cdot, E^{\cdot}_R(M))$ to this
short exact sequence of complexes. Since $E^{\cdot}_R(M)$ is a left bounded complex of
injective $R$-modules it yields  a short exact sequence of complexes of $R$-modules
\begin{gather*}
0\rightarrow \Hom_R(C^{\cdot}_M(I), E^{\cdot}_R(M))\to \Hom_R(\Gamma_I(E^{\cdot}_R(M)),
E^{\cdot}_R(M))\\
\to \Hom_R(H^c_I(M), E^{\cdot}_R(M))[c]\to 0.
\end{gather*}
First we investigate the complex in the middle.  There is a quasi-isomorphism of complexes
$\Gamma_I(E^{\cdot}_R(M)) \qism L_{\xx} \otimes E^{\cdot}_R(M)$ (see \cite[Theorem 1.1]{pS2}). That is the complex in the
middle is quasi-isomorphic to $X = \Hom_R(L_{\xx} \otimes E^{\cdot}_R(M),E^{\cdot}_R(M))$ since
$ E^{\cdot}_R(M)$ is a left bounded complex of injective $R$-modules. By adjunction (see Lemma \ref{4.5}) $X$ is
isomorphic to the complex $\Hom_R(L_{\xx}, \Hom_R( E^{\cdot}_R(M), E^{\cdot}_R(M)))$. Because
$ E^{\cdot}_R(M)$ is an injective resolution of $M$ so there is a quasi-isomorphism of complexes
\[
\Hom_R( E^{\cdot}_R(M), E^{\cdot}_R(M)) \qism \Hom_R(M, E^{\cdot}_R(M))
\]
Since $L_{\xx}$ is a bounded complex of free $R$-modules it follows that
$X$ is quasi-isomorphic to $\Hom_R(L_{\xx}, \Hom_R(M, E^{\cdot}_R(M)))$.
That is, in order to compute the homology of $X$ it will be enough to compute the homology
of $\Hom_R(L_{\xx}, \Hom_R(M, E^{\cdot}_R(M)))$. By virtue of Theorem \ref{2.3}
there is the following spectral sequence
\[
E_2^{i,j} = {\rm L}_{i}\Lambda^I(\Ext^j_R(M,M)) \Rightarrow
E^{i+j}_{\infty} = H^{i+j}(\Hom_R(L_{\xx},\Hom_R(M, E^{\cdot}_R(M))))
\]
It degenerates to isomorphisms ${\rm L}_{0}\Lambda^I(\Ext_R^j(M,M)) \cong H^j(\Hom_R(L_{\xx},\Hom_R(M, E^{\cdot}_R(M))))$ since $\Ext^j(M,M)$ is finitely generated for all $j \in \mathbb{Z}$.
This finally proves that $H^j(X) \cong {\rm L}_{0}\Lambda^I(\Ext_R^j(M,M))$ for all $j \in \mathbb{Z}$.

The cohomology of the complex at the right of the above short exact sequence of complexes
is $\Ext_R^{i+c}(H^c_I(M),M), i \in \mathbb{Z}$.
Therefore the long exact cohomology sequence at degree 0 provides a natural homomorphism
\[
\Hom_{\hat{R}^I}(\hat{M}^I,\hat{M}^I) \to \Ext_R^{c}(H^c_I(M),M).
\]
The second module is isomorphic to $\Hom_R(H^c_I(M), H^c_I(M))$ (see Proposition \ref{31} (a)).
This completes the proof.
\end{proof}

In the following we shall investigate the particular case when $H^i_I(M) = 0$ for
all $ i \not= c = \grade(I,M)$.

\begin{cor}\label{3a}
With the previous notation suppose in addition that $H^i_{I}(M)= 0$ for all $i\neq c$.
Then there are isomorphisms
\[
\Hom_{\hat{R}^I}(\hat{M}^I,\hat{M}^I) \cong \Hom_R(H^c_I(M), H^c_I(M))
\]
and
\[
\Ext^i_{\hat{R}^I}(\hat{M}^I,\hat{M}^I) \cong \Ext^{i+c}_R(H^c_I(M), M)\cong \Ext^{i}_R(H^c_I(M), H^c_I(M))
\]
for all $i \in \mathbb{Z}$.
\end{cor}

\begin{proof}  The assumption on the vanishing of $H^i_I(M)$ for all $i \not= c$ implies that
the truncation complex $C^{\cdot}_M(I)$ is homologically trivial. Therefore
$H^i(\Hom_R(C^{\cdot}_M(I), E^{\cdot}_R(M))) = 0$ for all $i \in \mathbb{Z}$.
That is,  the short exact sequence of complexes in the proof of Theorem
\ref{2} induces a quasi-isomorphism
\[
\Hom_R(\Gamma_I(E^{\cdot}_R(M)),E^{\cdot}_R(M))  \qism \Hom_R(H^c_I(M), E^{\cdot}_R(M))[c].
\]
With the computations on the cohomology (done in the proof of Theorem \ref{2}) it provides isomorphisms
\[
\Ext^i_{\hat{R}^I}(\hat{M}^I,\hat{M}^I) \cong \Ext^{i+c}_R(H^c_I(M), M)
\]
for all $i \in \mathbb{Z}$. This finishes the proof (see Proposition \ref{31}$(a)$ and Proposition \ref{ps2}).
\end{proof}

\begin{rem}\label{r1} For an arbitrary $R$-module $X$ there is a natural injective homomorphism
$X \to D(D(X))$. By applying $\Hom_R(X,\cdot)$ and by adjunction (see Lemma \ref{4.5}) there is a natural
injective homomorphism $\Hom_{R}(X, X)\to \Hom_{R}(D(X),D(X))$. Therefore by Theorem \ref{2} the natural
homomorphism $\Hom_{\hat{R}^I}(\hat{M}^I, \hat{M}^I) \to \Hom_{R}(D(H^c_{I}(M)),D(H^c_{I}(M)))$
induces a commutative diagram
\[
\begin{array}{ccc}
  \Hom_{\hat{R}^I}(\hat{M}^I, \hat{M}^I) &  \to & \Hom_{R}(H^c_{I}(M),H^c_{I}(M)) \\
  \parallel &   & \downarrow \\
   \Hom_{\hat{R}^I}(\hat{M}^I, \hat{M}^I) & \to & \Hom_{R}(D(H^c_{I}(M)),D(H^c_{I}(M)))
\end{array}
\]
\end{rem}

The following lemma is a generalization of \cite[Corollary 2.4 and Theorem 2.6]{h1} to a finitely generated module.

\begin{lem}\label{3}
Let $R$ be a complete local ring of dimension $n$ and let $I\subset R$ be an ideal. Let $M$ denote a finitely generated
$R$-module and $\grade(I,M)= c$. Then the following holds:
\begin{itemize}
  \item [(a)] There is an isomorphism
\[
\Hom_{R}(H^c_{I}(M),H^c_{I}(M))\cong \Hom_{R}(D(H^c_{I}(M)),D(H^c_{I}(M))).
\]

 \item [(b)] The natural homomorphism
\[
\Hom_{R}(M,M)\to \Hom_{R}(H^c_{I}(M),H^c_{I}(M))
\]
is an isomorphism if and only if the natural homomorphism
\[
\Hom_{R}(M,M)\to \Hom_{R}(D(H^c_{I}(M)),D(H^c_{I}(M)))
\]
is an isomorphism.
\end{itemize}
\end{lem}

\begin{proof} First of all recall the well known fact that $R$ is $I$-adic
complete provided it is complete  (i.e. complete in the $\mathfrak{m}$-adic
topology).
By adjunction there is an isomorphism
\[
\Hom_{R}(D(H^c_{I}(M)),D(H^c_{I}(M)))\cong D(H^c_{I}(M)\otimes_R D(H^c_{I}(M))).
\]
The module on the right side is isomorphic to
\[
D(\Tor_c^R(H^c_{I}(M),D(M)))\cong \Ext^c_R(H^c_{I}(M), M).
\]
This follows by view of Proposition \ref{31} because  $R$ is complete and by Matlis duality the double Matlis dual of $M$ is isomorphic to $M$ since $M$ is finitely generated.
By virtue of Proposition \ref{31} there is the following isomorphism
\[
\Ext^c_R(H^c_{I}(M), M)\cong \Hom_{R}(H^c_{I}(M),H^c_{I}(M)).
\]
By combining these isomorphisms together it proves $(a)$. The statement $(b)$ is clear from $(a)$
and the commutative diagram in Remark \ref{r1}.
\end{proof}

\section{Applications}
As before let $(R,\mathfrak{m})$ be a ring and let $M$ denote a finitely generated $R$-module.
The following Theorem is actually proved by Schenzel (see \cite[Theorem 1.2]{p7}) in case of a complete Gorenstein local ring. Here we will give a generalization to a finitely generated $R$-module $M$.

\begin{thm}\label{21.3}
Let $R$ be a local ring of dimension $n$ and let $M$ be a finitely generated $R$-module. Let $J\subset I$ denote two ideals of $R$ such that $\grade(I,M)=c= \grade(J,M)$.
\begin{itemize}
  \item [(a)] There is a natural homomorphism
\[
\Hom_{R}(H^c_{J}(M),H^c_{J}(M))\to \Hom_{R}(H^c_{I}(M),H^c_{I}(M))
\]
\item [(b)] Suppose that $\Rad IR_{\mathfrak{p}}= \Rad JR_{\mathfrak{p}}$ for all ${\mathfrak{p}}\in V(J)\cap \Supp_R(M)$ such that $\depth_{R_\mathfrak{p}} (M_{\mathfrak{p}})\leq c$. Then the above natural homomorphism is an isomorphism.
\item [(c)] Suppose in addition that $R$ is complete. Then there is a natural homomorphism
\[
\Hom_{R}(D(H^c_{J}(M)),D(H^c_{J}(M)))\to \Hom_{R}(D(H^c_{I}(M)),D(H^c_{I}(M)))
\]
and under the assumption of $(b)$ it is an isomorphism.
\end{itemize}
\end{thm}

\begin{proof}
Since $J\subset I$ it induces the following short exact sequence
\[
0\to I^{\alpha}/J^{\alpha}\to R/J^{\alpha}\to R/I^{\alpha}\to 0
\]
for each integer $\alpha \geq 1$. Moreover note that $\grade(I^{\alpha}/J^{\alpha},M)\geq c$ (see e.g. Proposition \ref{31} $(a)$). Then the application of the functor $\Hom_{R}(\cdot, M)$ to the last sequence yields the following exact sequence
\[
0\to \Ext^c_R(R/I^{\alpha}, M)\to \Ext^c_R(R/J^{\alpha}, M)\to \Ext^c_R(I^{\alpha}/J^{\alpha},M).
\]
Now the direct limit is an exact functor. So pass to the direct limit of this sequence we get that
\begin{equation*}
0\to H^c_{I}(M)\to H^c_{J}(M)\mathop\to\limits^f \lim\limits_{\longrightarrow} \Ext^c_R(I^{\alpha}/J^{\alpha},M)
\end{equation*}
Let $N:= \im f$. Then this induces the short exact sequence
\[
0\to H^c_{I}(M)\to H^c_{J}(M)\to N\to 0
\]
Again by applying the functor $\Hom_{R}(\cdot, M)$ to this sequence we get the following homomorphism
\[
\Ext^c_R(H^c_{J}(M), M)\to \Ext^c_R(H^c_{I}(M), M).
\]
Then from Proposition \ref{31} $(a)$ the statement $(a)$ is shown.

For the proof of $(b)$ note that the local cohomology is independent up to the radical. So without loss of generality we may assume that $IR_{\mathfrak{p}}= JR_{\mathfrak{p}}$ for all ${\mathfrak{p}}\in V(J)\cap \Supp_R(M)$ such that
\[
\depth_{R_\mathfrak{p}} (M_{\mathfrak{p}})\leq c.
\]
We claim that $\Ext^c_R(I^{\alpha}/J^{\alpha},M)= 0$ for all $\alpha \geq 1$. Because  $\Ann_R(I^{\alpha}/J^{\alpha})= J^{\alpha}:_R I^{\alpha}$  it will be enough to prove that $\grade (J^{\alpha}:_R I^{\alpha}, M)\geq c+1.$ It is well-known that
\[
\grade (J^{\alpha}:_R I^{\alpha},M)= \inf \{\depth_{R_\mathfrak{p}} (M_{\mathfrak{p}}): {\mathfrak{p}}\in V(J^{\alpha}:_R I^{\alpha})\cap \Supp_R(M)\}
\]
(see e.g. \cite[Proposition 1.2.10]{her}). Suppose on contrary that the claim is not true. Then there exists a prime ideal
\[
{\mathfrak{p}}\in V(J^{\alpha}:_R I^{\alpha})\cap \Supp_R(M)
 \]
such that $\depth_{R_\mathfrak{p}} (M_{\mathfrak{p}})\leq c$. Moreover $\Supp_R(R/(J^{\alpha}:_R I^{\alpha}))=V(J^{\alpha}:_R I^{\alpha})$ is contained in $V(J)$. This implies that ${\mathfrak{p}}\in V(J)\cap \Supp_R(M)$ and $R_{\mathfrak{p}}/(J_{\mathfrak{p}}^{\alpha}R_{\mathfrak{p}}:_{R_{\mathfrak{p}}} I_{\mathfrak{p}}^{\alpha}R_{\mathfrak{p}})\neq 0$ which is a contradiction to our assumption
$IR_{\mathfrak{p}} = JR_{\mathfrak{p}}$ . So we have $\Ext^c_R(I^{\alpha}/J^{\alpha},M)= 0$ which shows that $(b)$ is true by virtue of the exact sequence above.

Finally the proof of (c) is a consequence of Lemma \ref{3} (b).
\end{proof}

\begin{exmp} \label{ps1}
For instance, if $J = (x_1,\ldots,x_c)R$ is generated by an $M$-regular sequence it follows
that $H^i_J(M) = 0$ for all $i \not= c$.  Therefore the natural homomorphism
\[
\Hom_{\hat{R}^J}(\hat{M}^J, \hat{M}^J)   \to  \Hom_{R}(H^c_{J}(M),H^c_{J}(M))
\]
is an isomorphism (see Theorem \ref{2}). By view of Theorem \ref{21.3} this could be
a candidate for proving that
\[
\Hom_{\hat{R}^I}(\hat{M}^I, \hat{M}^I) \to  \Hom_{R}(H^c_{I}(M),H^c_{I}(M))
\]
is an isomorphism where $J \subset I$. At least if $R$ is $I$-adic complete.
\end{exmp}

In the following we give a non-trivial example of Corollary \ref{3a}.

\begin{exmp} \label{ps3}
Let $R= A/J$ where $A= k[|x,y,z,w|]$ be the formal power series ring over a field $k$ and $J=(xz,yw) =(x,y)\cap (y,z)\cap (z,w)\cap (x,w)$. Clearly $R$ is a complete local Gorenstein ring of $\dim (R)= 2$. Suppose that $I=(x,y,z)R$. Let $M= R/\mathfrak{p}_1\oplus R/\mathfrak{p}_2$ where $\mathfrak{p}_1= (x,y)R$, and $\mathfrak{p}_2= (y,z)R$. Then it follows by Independence of Base Theorem (see \cite[Proposition 2.14]{h}) that for each $i\in \mathbb{Z}$
\[
H^i_I(R/\mathfrak{p}_1)\cong H^i_{IR/\mathfrak{p}_1}(R/\mathfrak{p}_1) \text{ and } H^i_I(R/\mathfrak{p}_2)\cong H^i_{IR/\mathfrak{p}_2}(R/\mathfrak{p}_2).
\]
Moreover $\cd(I,M)= \max \{\cd(I,R/\mathfrak{p}): \mathfrak{p}\in \Ass_R(M)\}$ it induces that $\cd(I,M)= 1$. Here $\cd(I,M)$ denotes the cohomological dimension of $M$ with respect to $I$. Also $x+z\notin \Ass_R(M)$ so $\grade(I,M)= 1$. That is $H^i_I(M)= 0$ for all $i\neq 1.$ Then by Corollary \ref{3a} the natural homomorphism
\[
\Hom_{R}(M,M)\to \Hom_{R}(H^c_{I}(M),H^c_{I}(M))
\]
is an isomorphism.

\end{exmp}

\begin{cor}\label{0}
Let $(R,\mathfrak{m})$ denote a complete local ring of $\dim(R)=n$. Let $M$ be a finitely generated $R$-module.
Let $J\subseteq I$ be two ideals and  $H^i_{J}(M)= 0$ for all $i\neq c= \grade(J,M)= \grade(I,M)$.
Suppose that $\Rad IR_{\mathfrak{p}}= \Rad JR_{\mathfrak{p}}$ for all prime ideals
${\mathfrak{p}}\in V(J)$ with $\depth_{R_\mathfrak{p}} (M_\mathfrak{p})\leq c$. Then the natural homomorphism
\[
 \Hom_{R}(M, M) \to \Hom_{R}(H^c_{I}(M),H^c_{I}(M))
\]
is an isomorphism.
\end{cor}
\begin{proof}
It follows from Theorem \ref{21.3} and Corollary  \ref{3a}.
\end{proof}


\end{document}